\documentclass[a4paper,12pt,oneside]{article}

\usepackage[utf8]{inputenc}
\usepackage[english]{babel}
\usepackage{amsthm}
\usepackage{amssymb}
\usepackage{amsmath}
\usepackage{hyperref}
\usepackage{mathrsfs}

\usepackage{indentfirst}
\usepackage{pstricks}

 \usepackage{tikz}

\usepackage{graphicx}
\graphicspath{{images/}}

\usepackage{longtable,geometry}
\usepackage{color}

\usepackage{hyperref}

\newtheorem{theorem}{Theorem}[section]

\newtheorem{lemma}[theorem]{Lemma}
\newtheorem*{remark}{Remark}
\newtheorem{definition}{Definition}[section]
\newtheorem{proposition}{Proposition}[section]

\usepackage{wasysym} 

\title{Newhouse Laminations of polynomials on $\mathbb{C}^2$ }

\author{Marco Martens, Liviana Palmisano and Zhuang Tao}

\begin{document}
\bibliographystyle{plain}
\maketitle
\begin{abstract}
It has been recently discovered that in smooth unfoldings of maps with a rank-one homoclinic tangency there are codimension two laminations of maps with infinitely many sinks. Indeed, these laminations, called Newhouse laminations, occur also in the holomorphic context. In the space of polynomials of $\mathbb{C}^2$, with bounded degree, there are Newhouse laminations. 
\end{abstract}
\section{Introduction}
The theory of dynamical systems pretends to contribute to the study of real 
world systems. For this reason one expect that systems modeling natural 
processes have a certain form of stability, otherwise their behavior would 
not be observed.

The relevant aspects of a dynamical system are found in the attractor of the system, i.e.,the set where most orbits spent most of the time. Hence, one has to study the attractors of a systems as well as their stability. Attracting periodic points, i.e., sinks, are the simplest attractors exhibiting the strongest form of stability. Indeed, a sink persists in an open neighborhood. Moreover when the period of the sink becomes higher one expects that the neighborhood where the sink survives becomes smaller. 
If then one considers a map with infinitely many sinks with arbitrarily high periods, there is no reason to belief that they are simultaneously stable in any sense. Indeed Newhouse constructed maps with infinitely many sinks, see \cite{Newhouse, PT}.

However, in \cite{1}, it has been shown that in smooth unfoldings of maps with a rank one homoclinic tangency, there is a codimension two lamination of maps with infinitely many sinks. Surprisingly, the attracting sinks, all together, with their own topology, survive along the leaves of a lamination. They are stable in this sense.

A natural question is whether all relevant attractors  have a similar form of stability: do they survive along finite codimension manifolds? Indeed this is true in one-dimensional dynamics. For example, in the space of circle diffeomorphisms the topological classes of non periodic attractors are codimension one manifolds. In the higher dimensional setting this question is still completely open. However there are example of non periodic attractors which also survive along finite codimension manifolds, see \cite{CLM, Pa}.

The present paper is a continuation in the study of stability of attractors in the holomorphic setting. Namely there are Newhouse laminations in the space of polynomials of ${\mathbb{C}^2}$.
\vskip .2 cm
\paragraph{\bf Theorem A.}
The space $\text{Poly}_d({\mathbb{C}^2})$ of complex polynomials of $\mathbb{C}^2$ of degree at most $d$, with $d\ge 2$, contains a codimension $2$ lamination of maps with infinitely many sinks. The lamination is homeomorphic to the Baire set times $\mathbb D^{\text{D}-2}$ where $\text{D}$ is the dimension of $\text{Poly}_d({\mathbb{C}^2})$. The leaves of the lamination are holomorphic. The sinks persist along each leave of the lamination.
\vskip .2 cm
In particular, any generic three dimensional family of polynomials will cross the lamination in Theorem $A$. A specific example is described in the following theorem.
\vskip .2 cm

\paragraph{\bf Theorem B.}
Any holomorphic family $F:\mathbb C\times\mathbb C\times \mathbb D^{T}\to \text{Poly}_d({\mathbb{C}^2})$, with $T\geq 1$, containing the H\'enon family, i.e.
$$
F_{a,b,0}\left(\begin{matrix}
x\\y
\end{matrix}
\right)=\left(\begin{matrix}
a-x^2-by\\x
\end{matrix}
\right)
$$
contains a codimension $2$ lamination of maps with infinitely many sinks. The lamination is homeomorphic to the Baire set times $\mathbb D^{T}$. The leaves of the lamination are holomorphic. The sinks persist along each leave of the lamination.
\vskip .2 cm

Theorem $A$ and Theorem $B$ are examples of a more general theorem. Newhouse laminations exist in any holomorphic family unfolding a polynomial with a strong homoclinic tangency, see Definition \ref{unfolding}.

\vskip .2 cm

\paragraph{\bf Theorem C.} 
Let $F_{0,0,0}\in \text{Poly}_d({\mathbb{R}^2})$ be a polynomial with a strong homoclinic tangency and $F:\mathbb D\times\mathbb D\times \mathbb D^T\to \text{Poly}_d({\mathbb{C}^2})$, $T\geq 1$, an holomorphic family which unfolds $F_{0,0,0}$. Then, in $\mathbb D\times\mathbb D\times \mathbb D^T$ there exists a codimension $2$ lamination of maps with infinitely many sinks. The lamination is homeomorphic to the Baire set times $\mathbb D^{T}$. The leaves of the lamination are graphs of holomorphic functions over $\mathbb D^{T}$. The sinks persist along each leave of the lamination.
\vskip .2 cm

The existence of polynomial maps with infinitely many sinks has been previously proved. For example, in \cite{Bu}, the author constructs polynomial maps of $\mathbb C^2$ which have infinitely many sinks. These examples are polynomial with very large degree. The assumption on the degree has been dropped in \cite{Bi} , where the author shows the existence of polynomial of degree $d$, $d\geq 2$ of $\mathbb{C}^{3}$ which have infinitely many sinks. The same statement has been proved for holomorphic maps of $\mathbb{P}^{2}$ in \cite{5}. Here, we cover the case of polynomials of $\mathbb{C}^2$ of any degree. Furthermore, the stability of these maps has been analyzed.  In particular, Theorem $A$ implies that there are arbitrarly high dimensional holomorphic families of polynomial maps for which every map in the family has infinitely many sinks. This answers Question $2.1$ in \cite{6}.

We conclude with some words about the proofs of our theorems. In \cite{1} it has been shown that laminations of maps with infinitely many sinks exist in unfoldings of polynomials of ${\mathbb{R}^2}$. These laminations in $\text{Poly}_d\left({\mathbb{R}^2}\right)$ are in fact restrictions to the real slice of laminations in $\text{Poly}_d\left({\mathbb{C}^2}\right)$. In this paper we show that the maps in the extension also have infinitely many sinks. 

We characterize sinks by their trace and the Jacobian. This is why our theorems are valid in $\text{Poly}_d\left({\mathbb{C}^2}\right)$. We believe that the same holds for polynomials in $\mathbb C^n$.

\section{Preliminaries}
In this section we collect all definitions and relevant statements from \cite{1} needed to prove the main theorems.
The discussion begins with real polynomials of $\mathbb R^2$.

\bigskip

The following well-known linearization result is due to Sternberg.
\begin{theorem}\label{Ctlinearization}
Given $\left(\lambda,\mu\right)\in\mathbb R^{2}$, there exists $N\left(\lambda,\mu\right)\in\mathbb N$ such that the following holds. 
Let $f:\mathbb R^{2}\to \mathbb R^{2}$ be a polynomial with saddle point $p\in \mathbb R^{2}$ having unstable eigenvalue $|\mu|>1$ and stable eigenvalue $\lambda$. If
\begin{equation}\label{nonresonance}
\lambda\neq\mu^{k_1} \text{ and } \mu\neq\lambda^{k_2}
\end{equation}
for $k=\left(k_1,k_2\right)\in\mathbb N^{2}$, with $2\leq |k|=k_1+k_2\leq N$ and $N$ large enough, then $f$ is $\mathcal{C}^{4}$ linearizable.
\end{theorem}
\begin{definition}
 We say that $p$ satisfies the $\mathcal{C}^{4}$ non-resonance condition if (\ref{nonresonance}) holds.
\end{definition}

\begin{theorem}\label{familydependence}
Let $f:\mathbb R^{2}\to \mathbb R^{2}$ be a polynomial with saddle point $p\in \mathbb R^{2}$ which satisfies the $\mathcal{C}^{4}$ non-resonance condition. Let $0\in \mathcal P\subset\mathbb R^{n}$ and let $F:\mathbb R^{2}\times \mathcal P\to \mathbb R^{2}$ be a $\mathcal{C}^{\infty}$ family with $F_0=f$. Then, there exists a neighborhood $U$ of $p$ and a neighborhood $V$ of $0$ such that, for every $t\in V$, $F_t$ has a saddle point $p_t\in U$ satisfying the $\mathcal{C}^{4}$ non-resonance condition. Moreover $p_t$ is $\mathcal{C}^{4}$ linearizable in the neighborhood $U$ and the linearization depends $\mathcal{C}^{4}$ on the parameters. 
\end{theorem}
The proofs of Theorem \ref{Ctlinearization} and Theorem \ref{familydependence} can be found in \cite{BrKo, IlaYak}. In the sequel we introduce the concept of a map with a strong homoclinic tangency which appears already in $\S 2$ of \cite{1}.

\begin{definition}
Let $f: \mathbb R^2 \rightarrow \mathbb R^2$ be a polynomial which is a local diffeomorphism satisfying the following conditions:\\
$(f1)$ $f$ has a saddle point $p \in \mathbb R^2$, with unstable eigenvalue $\mu$ and stable eigenvalue $\lambda$,\\
$(f2)$ $\left|\lambda \| \mu\right|^{3}<1$,\\
$(f3)$ $p$ satisfies the $\mathcal{C}^{4}$ non-resonance condition,\\
$(f4)$ $f$ has a non degenerate homoclinic tangency, $q_1\in W^u(p)\cap W^s(p)$,\\ 
$(f5)$$f$ has a transversal homoclinic intersection,  $q_2\in W^u(p)\pitchfork W^s(p)$, 
\\
$(f6)$ let $[p,q_2]^u\subset W^u(p)$ be the arc connecting $p$ to $q_2$, then there exist arcs  $W^u_{\text{\rm loc},n}(q_2)=[q_2, u_n]^u\subset W^u(q_2)$ such that $[p,q_2]^u\cap [q_2,u_n]^u=\left\{q_2\right\}$ and 
$$
\lim_{n\to\infty}f^n\left(W^u_{\text{\rm loc},n}(q_2)\right)=[p,q_2]^u,
$$\\
$(f7)$ there exist neighborhoods $W^u_{\text{\rm loc},n}(q_1)\subset W^u(q_1)$ such that 
$$
\lim_{n\to\infty}f^n\left(W^u_{\text{\rm loc},n}(q_1)\right)=[p,q_2]^u,
$$\\
$(f8)$ there exists $N\in\mathbb N$ such that 
$$
f^{-N}(q_1)\in [p,q_2]^u.
$$
A map $f$ with these properties is called a map with a strong homoclinic tangency.
\end{definition}
\begin{remark} If the unstable eigenvalue is negative, $\mu<-1$, then $(f6)$, $(f7)$, and $(f8)$ are redundant.
\end{remark}
\begin{remark} As shown in the proof of Theorem $C$ in \cite{1}, real polynomials with a strong homoclinic tangency exists.
\end{remark}
Next, following \cite{1}, we introduce the concept of unfolding of a map with a strong homoclinic tangency.

\bigskip 

Let $\mathcal{P}=[-r, r] \times[-r, r]$ with $r>0$. Given a map $f$ with a strong homoclinic tangency, we consider a $\mathcal{C}^{\infty}$ family $F: \mathcal{P} \times \mathbb R^2 \rightarrow \mathbb R^2$ through $f$ with the following properties:\\
$(F1)$ $F_{0,0}=f$,\\
$(F2)$ $F_{t, a}$ has a saddle point $p(t, a)$ with unstable eigenvalue $|\mu(t, a)|>1,$ stable eigenvalue $\lambda(t, a),$ and
$$
\frac{\partial \mu}{\partial t} \neq 0,
$$
$(F3)$ let $\mu_{\max }=\max _{(t, a)}|\mu(t, a)|, \lambda_{\max }=\max _{(t, a)}\left|\lambda(t, a)\right|$ and assume
$$
\lambda_{\max } \mu_{\max }^{3}<1,
$$
$(F4)$ there exists a $\mathcal{C}^{2}$ function $[-r, r] \ni t \mapsto q_{1}(t) \in W^{u}(p(t, 0)) \cap W^{s}(p(t, 0))$ such that $q_{1}(t)$ is a non degenerate homoclinic tangency.

\bigskip 
According to Theorem \ref{familydependence} we may make a change of coordinates to ensure  that the family $F$ is $\mathcal C^4$ and that, for all $(t,a)\in [-r_0,r_0]^{2}$ with $0<r_0<r$, $F_{t,a}$ is linear on the ball $[-2,2]^2$, namely $$F_{t,a}=\left(\begin{matrix}
\lambda(t,a)&0\\
0&\mu(t,a)\\
\end{matrix}\right).$$ 
Moreover, the saddle point $p(t,a)=(0,0)$ and the local stable and unstable manifolds satisfy:
\begin{itemize}
\item[-] $W^s_{\text{loc}}(0)=[-2,2]\times \left\{0\right\}$,
\item[-] $W^u_{\text{loc}}(0)=\left\{0\right\}\times [-2,2]$,
\item[-] $q_1(t)\in (0,1]\times \left\{0\right\}\subset W^s_{\text{loc}}(0)$,
\item[-] $q_2(t,a)\in \left\{0\right\}\times \left(\frac{1}{\mu},1\right)\subset W^u_{\text{loc}}(0)$,
\item[-] there exists $N$ such that $f^{N}(q_3(t))=q_1(t)$ where $q_3(t)=(0,1)$,
\item[-]$Df^N_{q_3}(e_1)\notin T_{q_1}W^s(0)
$
and it points in the positive $y$ direction.
\end{itemize}

\bigskip

The next lemma states that $q_3$ is contained in a curve of points whose vertical tangent vectors are mapped by $DF^{N}$ to horizontal ones. The proof is the same as the one of Lemma $2$ in \cite{1}. 
Let $(x,y)$ be in a neighborhood of $q_3$ and consider the point $$(X_{t,a}(x,y),Y_{t,a}(x,y))=F^N_{t,a}(x,y).$$
\begin{lemma}\label{functionc}

There exist $x_0, a_0>0$, a $\mathcal C^2$ function $c:[-x_0,x_0]\times [-t_0,t_0]\times  [-a_0,a_0]\to\mathbb R $ and a positive constant $Q$ such that 
$$\frac{\partial Y_{t,a}}{\partial y}\left(x, c(x,t,a)\right)=0,$$
and 
$$\frac{\partial^2 Y_{t,a}}{\partial y^2}\left(x, c(x,t,a)\right)\geq Q.$$
\end{lemma}

\begin{definition}\label{criticalpointcriticalvalue}
Let $\left(t,a\right)\in [-t_0,t_0]\times  [-a_0,a_0]$. We call the point $$c_{t,a}=\left(0, c(0,t,a)\right)$$ the {primary critical point} and 
 $$z_{t,a}=F_{t,a}^N\left(c_{t,a}\right)=(z_x(t,a),z_y(t,a))$$ the {primary critical value} of $F_{t,a}$.
\end{definition}
We are now ready for the definition of unfolding of a map with a strong homoclinic tangency.

\begin{definition}\label{unfolding}
A family $F_{t,a}$ is called an \emph{unfolding} of $f$ if it can be reparametrized such that\\
$(P1)$ $z_y(t,0)=0$,\\
$(P2)$ $\frac{\partial z_y(t,0)}{\partial a}\neq 0.$ 
\end{definition}
\begin{remark}\label{rem:zyhighta}
Without lose of generality we may assume that if $F$ is an unfolding then  $z_y(t,a)=a$, the primary critical value is  at height $a$ and the primary critical point $c(t,a)=(0,1)$.
\end{remark}
Consider a polynomial map $f$ with a strong homoclinic tangency and an holomorphic family 
$$\mathbb{D}\times \mathbb{D}\ni (t,a)\mapsto F_{t,a}\in \text{Poly}_d({\mathbb{C}^2}).
$$
such that the real part $(-1,1)\times (-1,1)\ni (t,a)\mapsto F_{t,a}\in \text{Poly}_d({\mathbb{R}^2})
$ is an unfolding of $f$.
This holomorphic family is also called an unfolding. Assume  that the unfolding is contained in a larger polynomial family
$$\mathbb{D}\times \mathbb{D}\times \mathbb{D}^T\ni (t,a, \tau)\mapsto F_{t, a, \tau}\in \text{Poly}_d({\mathbb{C}^2}).
$$

There is a local holomorphic change of coordinates such that the saddle point becomes $(0,0)$, the local stable manifold contains the unit disc in the $x$-axis, and the local unstable manifold contains the unit disc in the $y$-axis. 
Moreover, the restriction of the map to the invariant manifolds is linearized, that is 
\begin{equation}\label{semilinearization}
F(x,0)=(\lambda x,0) \text{ and } F(0,y)=(0, \mu y).
\end{equation}
The domain $\mathbb{D}\times \mathbb{D}$ where (\ref{semilinearization}) holds, is called the domain of semi-linearization.
The change of coordinates depends holomorphically on the parameters.  
Observe that in the domain of semi-linearization $F$ is not necessarily linear. Moreover, for $(x,y)$ in the domain of semi-linearization we have the following estimate 
\begin{equation}\label{eq:fxyinpspu}
 F(x, y) = (\lambda x + P_{s}(x, y), \mu y + P_{u}(x, y)),  
\end{equation}
where $P_{s}$ and $ P_{u}$ are holomorphic functions satisfying $P_s(x,0)=0$, $P_s(0,y)=0$, $P_u(x,0)=0$ and $P_u(0,y)=0$. Their derivatives at the origin are zero. This implies, for $(x,y)\in \mathbb{D}\times \mathbb{D}$, 
\begin{equation}\label{DFnCC}
DF(x,y)=\left(\begin{matrix}
\lambda+O(y) & O(x)\\
O(y) & \mu +O(x)
\end{matrix}\right).
\end{equation}

Choose a parameter $(t,a,\tau)$ and assume that there is a periodic point $p$ in the domain of semi-linearization which returns in $N$ steps into the domain of semi-linearization and then needs $n$ steps inside to return to itself. Let $$
(t,a, \tau)\mapsto \text{tr} DF^{N+n}_p.
$$
Observe that if $\text{\rm tr} DF^{N+n}_p=0$, then for $n\ge 1$ large enough,  the periodic orbit of $p$ is attractive, called {\it strong sink}. According to $\S 6$ in \cite{1}, for $n$ large enough, there exists an holomorphic function  
\begin{equation}\label{eq:sandescr}
sa_n: \mathbb{D}\times \mathbb{D}^T\to \mathbb{C}.
\end{equation}
with the following property. Along the graph of $sa_n$, in parameters of the form $(t,sa_n(t,\tau),\tau )$ there is a periodic point of period $N+n$ in the domain of semi-linearization which is a strong sink. Moreover, by Lemma $19$ in \cite{1},
\begin{equation}\label{eq:sanoforder1overmutothen}
sa_n=O\left(\frac{1}{\mu^n}\right).
\end{equation}
Besides, in \cite{1}, a holomorphic function $b_{n,n_{0}}$ has been constructed, 
\begin{equation}\label{eq:bnholomorphic}
b_{n,n_{0}}: \mathbb{D} \times \mathbb{D}^{T} \rightarrow \mathbb{C},
\end{equation}
 such that, along the graph of the map $b_{n,n_{0}}$, the map with parameter of the form $\left(t, b_{n, n_{0}}(t, \tau), \tau\right)$, has a non-degenerate homoclinic tangency of the original saddle. They are called secondary tangencies. We have the following proposition, see $\S 6$ in \cite{1}.
\begin{proposition}\label{prop:anandbnintersect}
The graphs of the functions $sa_{n}$ and $b_{n,n_{0}}$ intersect transversally. Moreover the intersection is the graph of an holomorphic function 
$$
\mathbb{D}^T\ni \tau\mapsto l_{n,n_0}(\tau)=(t(\tau), a(\tau)).
$$ 
\end{proposition}
In particular, in the parameter $(l_{n,n_0}(\tau), \tau)$ the map has a non-degenerate secondary tangency and a strong sink. Observe that the points $(t(\tau), a(\tau),\tau)$ are in the graph of $sa_n$ and, as was shown in $\S 6$ of \cite{1}, they are uniformly bounded. 
\section{Proof of the Main Results}

We follow the idea of Newhouse boxes in \cite{1} to extend to Newhouse tubes in our case. Let $f$ be a polynomial with a strong homoclinic tangency and $F_{t,a}$, $(t,a)\in \mathbb{D}\times \mathbb{D}$, an unfolding of $f$. We define the tube around the strong sink curve $sa_n$ in the parameter space to be:

$$\mathcal{HA}_{n}=\left\{(t, a) \in\mathbb{D}\times \mathbb{D}\left|\right.| a-sa_{n}(t) | \leq \frac{\epsilon_{0}}{\left|\mu\left(t, sa_{n}(t)\right)\right|^{2 n}}\right\},$$

where $\epsilon_{0}$ is a small constant which will be adjusted later. 
\begin{definition}
Let $(t,a)\in \mathcal{HA}_{n}$. A periodic point $p$ is called {simple}, if it is of period $N+n$ and it is in the domain of semi-linearization $\mathbb D\times\mathbb D$. Furthermore, $p$ returns after $N$ iterates in the domain of semi-linearization, $F_{t,a}^{N}(p)\in \mathbb D\times\mathbb D$, and remains in it for the following $n$ iterates.
\end{definition}
Namely, we are going to prove the following proposition: 

\begin{proposition}\label{thm:neat}
There exists an $\epsilon_0>0$ such that, for $n$ large enough and for any $(t,a)$ in $\mathcal{HA}_{n}$, the map $F_{t,a}$ has an attracting simple period point.  
\end{proposition}
The proof of Proposition \ref{thm:neat} needs some preparation. Assume that $\mathbb{D}\times \mathbb{D}$ is the semi-linearization domain of $F_{t,a}$.
\begin{lemma}
Suppose $(x_{i},y_{i})_{i=1,2}\in \mathbb{D} \times \mathbb{D}$, $(\tilde{x_{i}},\tilde{y_{i}}) = F(x_{i},y_{i}) \in \mathbb{D} \times \mathbb{D}$, denote $x^{'}=$ max $\{ |x_{1}|, |x_{2}|\}$, $y^{'}=$ max $\{ |y_{1}|, |y_{2}|\}$, then we have 
\begin{equation*}
    \left|\tilde{x}_{2}-\tilde{x}_{1}\right| \leq \left|\lambda+O\left(y^{\prime}\right)\right|\,\left| x_{2}-x_{1}\right|+\left|O\left(x^{\prime}\right)\right|\,\left| y_{2}-y_{1} \right|,
\end{equation*}
\begin{equation*}
  \left|\tilde{y}_{2}-\tilde{y}_{1}\right| \leq|O\left(y^{\prime}\right)|\,| x_{2}-x_{1}\left|+|\mu+O\left(x^{\prime}\right)\right|\,| y_{2}-y_{1} |  .
\end{equation*}
In particular we have 
\begin{equation}\label{eq:tildex1}
 \left|\tilde{x}_{1}\right| \leq|\lambda+O(y_{1})|\,|x_{1}| ,
\end{equation}
\begin{equation}\label{eq:tildey1}
 \left|\tilde{y}_{1}\right| \leq|\mu+O\left(x_{1}\right)|\,|y_{1} |.
\end{equation}
\end{lemma}
\begin{proof} 
Consider the straight line $L$ through $\left( x_{1}, y_{1}\right)$ and $\left( x_{2}, y_{2}\right)$:
\begin{equation*}
\begin{aligned}
L:[0,1] &\longrightarrow \mathbb{D} \times \mathbb{D}\\ t &\rightarrow \left((1-t) x_{1} + t x_{2}, (1-t) y_{1}+t y_{2}\right)=\left(L_{x}, L_{y}\right)
\end{aligned}
\end{equation*}
then, by using \eqref{DFnCC}, we have 
\begin{equation*}
\begin{aligned}
\left|\tilde{x}_{2}-\tilde{x}_{1}\right| &= \left| F_{x}(L(1))-F_{x}(L(0)\right| \\
&\leq \int_{0}^{1}\left|\frac{\partial F_{x}}{\partial t}|_{L( t)} \right| d t \\
&\leq \int_{0}^{1}\left|\frac{\partial F_{x}}{\partial x}|_{L( t)} \right| \cdot \left|\frac{\partial L_{x}}{\partial t}\right|+ \left|\frac{\partial F_{x}}{\partial y}|_{L( t)} \right| \cdot \left|\frac{\partial L_{y}}{\partial t}\right|d t \\
&\leq \left|\lambda+O\left(y^{\prime}\right)\right|\,\left| x_{2}-x_{1}\right|+\left|O\left(x^{\prime}\right)\right|\,\left| y_{2}-y_{1} \right|.
\end{aligned}
\end{equation*}
Similarly, by considering the $y$ coordinate, we get the second estimate. For the last two inequalities it suffices to set $(x_{2},y_{2})=(0,y_{1})$ and $(x_{1},y_1)=(x_{1},0)$. 
\end{proof}

Estimates for orbits under semi-linearization have been studied previously. See for example \cite{7}. For completeness, using the above lemma, we prove the following estimates for the norm of coordinates of points which remain in the domain of semi-linearization after $n$ iterates. 

\begin{lemma}\label{lem3.2}
Let $(x,y)\in  \mathbb{D} \times \mathbb{D}$ such that $F^{i}(x,y)\in  \mathbb{D} \times \mathbb{D}$ for $i \leq n$. Denote $(x_{k},y_{k})=F^{k}(x,y)$, then for $n$ large enough we have $x_{k}=O(|\lambda|^{k})$ and $y_{k}=O(|\mu|^{-(n-k)})$.
\end{lemma}
\begin{proof} 
By \eqref{eq:fxyinpspu}, we have 
\[
\begin{array}{l}
{x_{k+1}=\lambda x_{k}+P_{s}\left(x_{k}, y_{k}\right)=\lambda x_{k}+D_{k} x_{k}} \\
{y_{k+1}=\mu y_{k}+P_{u}\left(x_{k}, y_{k}\right)=\mu y_{k}+E_{k}y_{k}}
\end{array}
\]
where 
\begin{equation}\label{eq:dkekorderestimats}
\left|D_{k}\right|\leq M \left|y_{k}\right| \text{ and } \left|E_{k}\right|\leq M \left|x_{k}\right|
\end{equation}
for some constant $M$.  
Choose some positive number $s<min\{\frac{1}{2},|\mu|-1, 1-|\lambda|\}$. By shrinking the semi-linearization domain appropriately we may assume that \begin{equation}\label{eq:dkeks}
    \left|D_{k}\right|,\left|E_{k}\right| \leq s.
\end{equation}
Then we have a priori estimates for $x_{k},y_{k}$,
\begin{equation*}
\left|x_{k}\right|=\left|\lambda+D_{k-1}\right|\left|x_{k-1}\right|=\left|x_{0}\right| \prod_{i=0}^{k-1}\left|\lambda+D_{i}\right| \leq(|\lambda|+s)^{k}\left|x_{0}\right|,
\end{equation*}
\begin{equation*}
\left|y_{k}\right|=\frac{1}{\left|\mu+E_{k}\right|}\left|y_{k+1}\right|=\left|y_{n}\right| \prod_{j=k}^{n-1} \frac{1}{|\mu+E_{j}|} \leq \frac{1}{\left(|\mu|-s\right)^{n-k}}\left|y_{n}\right|.
\end{equation*}
Thus we have,
\begin{equation}\label{eq:xksumestimates}
\sum_{k=0}^{n}\left|x_{k}\right| \leq \sum_{k=0}^{n}(|\lambda|+s)^{k}\left|x_{0}\right| \leq \frac{1}{1-|\lambda|-s}\left|x_{0}\right|,
\end{equation}
and 
\begin{equation}\label{eq:yksumestimates}
\sum_{k=0}^{n}\left|y_{k}\right| \leq \sum_{k=0}^{n} \frac{1}{(|\mu |-s)^{n-k}}\left|y_{n}\right| \leq \frac{|\mu|-s}{|\mu|-s-1}\left|y_{n}\right|.
\end{equation}
Notice that the lemma is automatically true for $x_{0}$ and $y_{n}$. Then, for some $0\leq k\leq n$, we have   
\begin{equation}\label{eq:xkestimates}
\begin{aligned}
x_{k} &=\lambda\left(1+\frac{D_{k-1}}{\lambda}\right) x_{k-1} 
&=\lambda^{k} x_{0} \prod_{i=0}^{k-1}\left(1+\frac{D_{i}}{\lambda}\right),
\end{aligned}
\end{equation}
and 
\begin{equation}\label{eq:ykestimates}
\begin{aligned}
y_{k} &=\mu^{-1} \frac{1}{1+\frac{E_{k}}{\mu}} y_{k+1}
&=\mu^{-(n-k)} y_{n} \prod_{i=k}^{n-1}\left(\frac{1}{1+\frac{E_{i}}{\mu}}\right).
\end{aligned}
\end{equation}
Using \eqref{eq:xkestimates}, the fact $\ln(x)\leq x-1$ for $x>0$, \eqref{eq:dkekorderestimats} and \eqref{eq:yksumestimates}, we have
\begin{equation*}
\begin{aligned}
\ln \left|\frac{x_{k}}{\lambda^{k}}\right| & \leq \ln \left|x_{0}\right|+\sum_{i=0}^{k-1} \ln \left(1+\left|\frac{D_{i}}{\lambda}\right|\right) \leq \sum_{i=0}^{k-1} \frac{\left|D_{i}\right|}{|\lambda|} \\
& \leq \frac{M}{|\lambda|}\left(\sum_{k=0}^{n}\left|y_{k}\right|\right)\leq \frac{M}{|\lambda|}\left(\frac{|\mu|-s}{|\mu|-s-1}\right)
\end{aligned}.
\end{equation*}
Similarly, using \eqref{eq:ykestimates}, the fact $\ln(x)\leq x-1$ for $x>0$, \eqref{eq:dkeks}, \eqref{eq:dkekorderestimats} and \eqref{eq:xksumestimates}, we have
\begin{equation*}
\begin{aligned}
\ln \left|\frac{y_{k}}{\mu^{-(n-k)}}\right| & \leq \ln \left|y_{n}\right|+\sum_{i=k}^{n-1} \ln \frac{1}{\left|1+\frac{E_{i}}{\mu}\right|}  \leq  \sum_{i=k}^{n-1} \frac{{\left|E_{i}\right|}/{|\mu|}}{1-{\left|E_{i}\right|}/{|\mu|}}\\
& \leq\frac{1}{|\mu|-s} \sum_{i=0}^{n}\left|E_{i}\right| \leq \frac{M}{|\mu|-s} \left(\frac{1}{1-|\lambda|-s}\right).
\end{aligned}
\end{equation*}
Hence, there is a constant $C$ such that $\left|x_{k}\right| \leq C|\lambda|^{k}$ and $\left|y_{k}\right| \leq C|\mu|^{-(n-k)}$. 
\end{proof}

Combining Lemma \ref{lem3.2} and the proof of Lemma $17$ in \cite{1}, we get the following.
\begin{lemma} \label{lem:dfn}
If $(x, y) \in \mathbb{D} \times \mathbb{D}$  and $F^{i}(x, y) \in \mathbb{D} \times \mathbb{D}$,  for  $i \leq n$, then 
\begin{equation*}
D F^{n}(x, y)=\left(\begin{array}{cc}{a_{11} \lambda^{n} \mu^{n}} & {a_{12} \lambda^{n} \mu^{n}} \\ {a_{21}} & {a_{22} \mu^{n}}\end{array}\right)
\end{equation*}
where $a_{k l}$ are uniformly bounded holomorphic functions and $a_{22} \neq 0$ is uniformly away from zero. 
\end{lemma}
In the following lemma we prove that a periodic point of period $N+n$ in the semi-linearization domain converges to $q_3(t)=F^{-N}(q_1(t))$, where $q_1(t)$ is the homoclinic tangency in the $x$-axis.
\begin{lemma}\label{lem:ptendstoq3}
Let $(t,a_n)\in \mathcal{HA}_{n}$ and $p_n$ a simple periodic point. Then 
$$
\lim_{n\to\infty}p_n=q_3(t).
$$
\end{lemma}
\begin{proof}
Observe that, because $F_{t,a_n}^{N}(p_n)$ remains in the domain of semi-linearization for the following $n$ iterates, then, by Lemma \ref{lem3.2}, $$\left[F_{t,a_n}^{N}(p_n)\right]_y=O\left(\frac{1}{\mu^n}\right),$$
and $$\left[p_n\right]_x=O\left({\lambda^n}\right).$$ 
 In particular $p_n$ is exponentially close to $F_{t,a_n}^{-N}(\mathbb D\times\left\{0\right\})$. Moreover by \eqref{eq:sanoforder1overmutothen} and the fact that $(t,a_n)\in \mathcal{HA}_{n}$, we have $a_n=O\left({1}/{\mu^n}\right)$. Observe that, $F_{t,0}^{-N}(\mathbb D\times\left\{0\right\})$ intersects the $y$ axis, $\left\{0\right\}\times\mathbb D$, in $q_3(t)$. Because $q_1(t)$ is a non degenerate homoclinic tangency, we get that the distance between $F_{t,a_n}^{-N}(\mathbb D\times\left\{0\right\})$ and  $q_3(t)$ is $O\left({1}/{\mu^{n/2}}\right)$. The lemma follows.
\end{proof}
By a simple calculation one gets the following lemma.
\begin{lemma}\label{lem:attarctingperiodicpoint}
Let $(t,a)\in \mathcal{HA}_{n}$ and $p$ a simple periodic point.  If $\left|\operatorname{tr} D F_{p}^{N+n} \right|\leq{1}/{3}$
and
$
 \left|det D F_{p}^{N+n}\right|\leq{1}/{40}
$
then $p$ is an attracting periodic point.
\end{lemma}
Next we extend Lemma $18$ in \cite{1} by allowing the trace to be small but non zero. Observe that if $(t,a)\in \mathcal{HA}_{n}$ and $p$ is a simple periodic point with trace small enough, then, by the previous lemma, the periodic point is attracting and therefore it persists in a neighborhood of $(t,a)$. In particular, the partial derivative of the trace, with respect to $a$, is well defined. 
\begin{proposition}\label{prop:partialderivativestrace}
Choose $s$ large enough. There exists $K_s> 0$ such that the following holds. Let $(t,a)\in \mathcal{HA}_{n}$ and $p$ a simple periodic point.  If $$\left|\operatorname{tr} D F_{p}^{N+n} \right|\leq \left|\lambda\mu\right|^s$$ then, 
\begin{equation*}
\frac{1}{K_s}|\mu|^{2 n} \leq\left|\frac{\partial}{\partial a}\left(\operatorname{tr} D F_{p}^{N+n}\right)\right| \leq K_s|\mu|^{2 n}.
\end{equation*}
\end{proposition}
\begin{proof}
From Lemma \ref{lem:dfn} and using the assumption on the trace, we get
\begin{equation*}
D F_{p}^{N+n}=D F_{F^{N}(p)}^{n} D F_{p}^{N}=\left(\begin{array}{cc}{O\left(\left(\lambda \mu\right)^{n}\right)} & {O\left(\left(\lambda \mu\right)^{n}\right)} \\ {O\left(\mu^{n}\right)} & {O\left(\left(\lambda \mu\right)^{n}\right)}\end{array}\right)          \text{ if } s \geq n,
\end{equation*}
and 
\begin{equation*}
D F_{p}^{N+n}=D F_{F^{N}(p)}^{n} D F_{p}^{N}=\left(\begin{array}{cc}{O\left(\left(\lambda \mu\right)^{n}\right)} & {O\left(\left(\lambda\mu\right)^{n}\right)} \\ {O\left(\mu^{n}\right)} & {O\left(\left(\lambda \mu\right)^{s}\right)}\end{array}\right)  \text{ if } s < n.
\end{equation*}
Choose $s$ large but fixed. 
The periodic point $p = (p_x, p_y)$ has coordinates $(p_x, p_y) \in \mathbb D \times\mathbb  D$. We claim, by differentiating with respect to a the $x$-component of the equation $F^{N+n}(p_x, p_y) = (p_x, p_y)$ that
\begin{equation}\label{eq:dpyda}
\left(1+O\left(\left(\lambda \mu\right)^{n}\right)\right) \frac{\partial p_{x}}{\partial a}=O\left(\left(\lambda \mu\right)^{n}\right) \frac{\partial p_{y}}{\partial a}+\frac{\partial F_{x}^{N+n}}{\partial a}=O\left(\left(\lambda \mu\right)^{n}\right) \frac{\partial p_{y}}{\partial a}+O\left(\left(\lambda \mu\right)^{n}\right)
\end{equation}
Correspondingly, for the $y$-component, we claim
\begin{equation}\label{eq:dpxda}
\left(1+O\left(\left(\lambda \mu\right)^{s}\right)\right) \frac{\partial p_{y}}{\partial a}=O\left(\mu^{n}\right) \frac{\partial p_{x}}{\partial a}+\frac{\partial F_{y}^{N+n}}{\partial a}=O\left(\mu^{n}\right) \frac{\partial p_{x}}{\partial a}+K \mu^{n}
\end{equation}
where $K > 0$ is bounded away from zero. Observe that 
\begin{eqnarray*}
\frac{\partial F_{y}^{N}}{\partial a}(p)&=&\frac{\partial F_{y}^{N}}{\partial a}(p)(q_3,t,0)\\&+&\frac{\partial }{\partial x}\left(\frac{\partial F_{y}^{N}}{\partial a}\right)\left(p_x-q_{3,x}\right)+\frac{\partial }{\partial y}\left(\frac{\partial F_{y}^{N}}{\partial a}\right)\left(p_y-q_{3,y}\right)\\&+&\frac{\partial }{\partial a}\left(\frac{\partial F_{y}^{N}}{\partial a}\right)a=1+o(1)+O\left(\frac{1}{\mu^n}\right)
\end{eqnarray*}
where we used Remark \ref{rem:zyhighta}, Lemma \ref{lem:ptendstoq3} and \eqref{eq:sanoforder1overmutothen}.

The estimates for $\partial F_{x}^{N+n} / \partial a \text { and } \partial F_{y}^{N+n} / \partial a$
are obtained as follows. Observe that
\begin{equation*}
\begin{aligned} \frac{\partial F_{y}^{N+n}}{\partial a} &=a_{21} \frac{\partial F_{x}^{N}}{\partial a}(p)+a_{22} \mu^{n} \frac{\partial F_{y}^{N}}{\partial a}(p)+\frac{\partial F_{y}^{n}}{\partial a}\left(F^{N}(p)\right) \\ &=O(1)+a_{22}\left(1+o(1)\right) \mu^{n}+\frac{\partial}{\partial a} \int_{0}^{F_{y}^{N}(p)}\left[D F^{n}\left(F_{x}^{N}(p),y\right)\left(\begin{array}{c}{0} \\ {1}\end{array}\right)\right]_{y} d y \\ &=O(1)+a_{22}\left(1+o(1)\right) \mu^{n}+\frac{\partial}{\partial a} \int_{0}^{F_{y}^{N}(p)} a_{22}\left(F_{x}^{N}(p), y\right) \mu^{n} d y \\ &=O(1)+a_{22}\left(1+o(1)\right) \mu^{n}+O\left(n \mu^{n} F_{y}^{N}(p)\right) \\ &=K \mu^{n} \end{aligned}
\end{equation*}
where we used that \(F^{i}\left(F_{y}^{N}(p)\right)\) for \(i<n\) is in the domain of semi-linearization, namely \(F_{y}^{N}(p)=\)
\(O\left(1 / \mu^{n}\right)\) and \(O\left(n \mu^{n} F_{y}^{N}(p)\right)=O(n) \). Similarly,
\begin{equation*}
\begin{aligned} \frac{\partial F_{x}^{N+n}}{\partial a} &=a_{11}\left(\lambda \mu\right)^{n} \frac{\partial F_{x}^{N}}{\partial a}(p)+a_{12}\left(\lambda \mu\right)^{n} \frac{\partial F_{y}^{N}}{\partial a}(p)+\frac{\partial F_{x}^{n}}{\partial a}\left(F^{N}(p)\right) \\ &=O\left(\left(\lambda \mu\right)^{n}\right)+\frac{\partial}{\partial a}\left(\int_{0}^{F_{y}^{N}(p)} a_{12}\left(F_{x}^{N}(p), y\right)\left(\lambda \mu\right)^{n} d y+O(\lambda^{n} F_{x}^{N}(p))\right) \\ &=O\left(\left(\lambda \mu\right)^{n}\right)+O\left(n \lambda^{n}\right) \\ &=O\left(\left(\lambda \mu\right)^{n}\right) 
\end{aligned}
\end{equation*}
where we used that \(F_{y}^{N}(p)=O\left(1 / \mu^{n}\right)\). From \eqref{eq:dpxda}, \eqref{eq:dpyda} and the fact that $\lambda\mu^{2}<1$, we have
\begin{equation}\label{eq:dpydafinal}
\frac{1}{2 K}|\mu|^{n} \leq\left|\frac{\partial p_{y}}{\partial a}\right| \leq 2 K|\mu|^{n},
\end{equation}
and
\begin{equation}\label{eq:dpxdafinal}
\left|\frac{\partial p_{x}}{\partial a}\right|=O\left(\left(\lambda \mu^{2}\right)^{n}\right)
\end{equation}
Observe that, 
\begin{equation*}
\operatorname{tr} D F_{p}^{N+n}=\tilde{A}\left(p_{x}, p_{y}, t, a\right)\left(\lambda \mu\right)^{n}+a_{22}\left(p_{x}, p_{y}, t, a\right) D\left(p_{x}, p_{y}, t, a\right) \mu^{n}+a_{21} B\left(p_{x}, p_{y}, t, a\right)
\end{equation*}
where \(D\) is the entry \(\left(D F_{p}^{N}\right)_{22},\) \(B=\left(D F_{p}^{N}\right)_{12}\) and, by Lemma \ref{lem:dfn},
the factors \(\tilde{A}\) are \(\mathcal{C}^{1}\) uniformly bounded when \(n\) gets large. Hence,
\begin{equation}\label{eq:tracedfnN}
\begin{aligned} \frac{\partial}{\partial a}\left(\operatorname{tr} D F_{p}^{N+n}\right) &=\left[\frac{\partial \tilde{A}}{\partial x} \frac{\partial p_{x}}{\partial a}+\frac{\partial \tilde{A}}{\partial y} \frac{\partial p_{y}}{\partial a}+\frac{\partial \tilde{A}}{\partial a}\right]\left(\lambda \mu\right)^{n}+n \tilde{A}\left(\lambda \mu\right)^{n-1} \frac{\partial \lambda \mu}{\partial a} \\ &+\left[\frac{\partial\left(a_{22} D\right)}{\partial x} \frac{\partial p_{x}}{\partial a}+\frac{\partial\left(a_{22} D\right)}{\partial y} \frac{\partial p_{y}}{\partial a}+\frac{\partial\left(a_{22} D\right)}{\partial a}\right] \mu^{n} \\ &+n a_{22} D \mu^{n-1} \frac{\partial \mu}{\partial a}+\frac{\partial\left(a_{21} B\right)}{\partial a} \\ &=O\left(n|\mu|^{n}\right)+\frac{\partial\left(a_{22} D\right)}{\partial y} \frac{\partial p_{y}}{\partial a} \mu^{n} \end{aligned}
\end{equation}
Where we use \eqref{eq:dpxdafinal}, \eqref{eq:dpydafinal}. Observe that,
\begin{equation*}
\frac{\partial\left(a_{22} D\right)}{\partial y}=\frac{\partial a_{22}}{\partial y} D+a_{22} \frac{\partial D}{\partial y}
\end{equation*}
is bounded away from zero. First, $D$ tends to zero, because \(a_{22} D \mu^{n} + a_{21}B =O ((\lambda\mu)^s)-\tilde{A}\left(\lambda \mu\right)^{n}\), $a_{22}$ is bounded away from zero by Lemma \ref{lem:dfn}, and $ \partial D / \partial y$ is away from zero because the family is an unfolding of a non-degenerate tangency. The lemma follows from \eqref{eq:dpydafinal} and \eqref{eq:tracedfnN}.
\end{proof}

\begin{proof}[Proof of Proposition \ref{thm:neat}]
Let $(t,a_0)\in\text{graph}(sa_n)$, then $F_{t,a_0}$ has a simple periodic orbit $p_{t,a_0}$ with trace zero.  Let $\left[a_0-\Delta,a_0+\Delta\right]$, with $\Delta\leq\epsilon_0/\mu^{2n}$, be the maximal interval such that $F_{t,a}$, $a\in \left[a_0-\Delta,a_0+\Delta\right]$ has an attracting simple periodic point $p:=p_{t,a}$ with 
$$
\left|\operatorname{tr} D F_p^{N+n}\right|\ \leq \frac{1}{3}.
$$
Observe that $p$ depends holomorphically on $t$ and $a$. 
Moreover, by Proposition \ref{prop:partialderivativestrace},
\begin{equation*}
\frac{1}{K}|\mu|^{2 n} \leq\left|\frac{\partial}{\partial a}\left(\operatorname{tr} D F_{p}^{N+n}\right)\right| \leq K|\mu|^{2 n}
\end{equation*}
where K is a positive constant.  Hence, for every $a\in \left[a_0-\Delta,a_0+\Delta\right]$, 
$$
\left|\operatorname{tr} D F_p^{N+n}\right|\leq K\mu^{2n}\Delta.
$$
If $\epsilon_0$ is small enough, then $\Delta=\epsilon_0/\mu^{2n}$. The proposition follows by applying Lemma \ref{lem:attarctingperiodicpoint}.
\end{proof}
 We are now ready to prove our main theorems. Theorem $A$ and Theorem $B$ are a direct consequence of Theorem $C$. Indeed the families in Theorems $A$ and $B$ are a concrete example of the unfoldings in Theorem $C$. It is then enough to prove Theorem $C$. 

\begin{proof}[Proof of Theorem $C$]
Let $F_{0,0,0}\in \text{Poly}_d({\mathbb{R}^2})$ be a polynomial with a strong homoclinic tangency and $F:\mathbb D\times\mathbb D\times \mathbb D^T\to \text{Poly}_d({\mathbb{C}^2})$, $T\geq 1$, an holomorphic family which unfolds $F_{0,0,0}$. Consider the holomorphic functions, 
$$
b_{n,n_{0}}: \mathbb{D} \times \mathbb{D}^{T} \rightarrow \mathbb{C},
$$
and 
$$
sa_{n}: \mathbb{D} \times \mathbb{D}^{T} \rightarrow \mathbb{C},
$$
as defined in \eqref{eq:bnholomorphic} and \eqref{eq:sandescr}. By Proposition \ref{prop:anandbnintersect}, the graphs of these functions intersect transversally in the graph of the function 
$$
 l_{n,n_0}:\mathbb{D}^{T} \rightarrow \mathbb{C}\times \mathbb{C}.
$$
Choose $w>0$ small. For a given $\tau\in \mathbb{D}^{T} $ define 
$$
P_{n,n_0}(\tau)=\left\{(t,a,\tau)\left|\right.\left| t-\left[l_{n,n_0}(\tau)\right]_x\right|\leq w, \left| a-b_{n,n_0}(t,\tau)\right|\leq w\right\}.
$$
Let $\mathcal {HA}_n(\tau)$ be the sink strip of the family $(t,a)\to F_{t,a,\tau}$.

For $w$ small enough, we can apply the same argument as in the proof of Proposition $5$ in \cite{1}, and we get that the restricted family $F:P_{n,n_0}(\tau)\to \text{Poly}_d({\mathbb{C}^2})$ is an unfolding of the map $F_{l_{n,n_0}(\tau),\tau}$ in the intersection of the graphs of $sa_n$ and $b_{n,n_0}$. In fact, for real $\tau$, the map $F_{l_{n,n_0}(\tau),\tau}$ has a strong homoclinic tangency, namely the secondary tangency. 

Moreover, by the same argument as in the proof of Proposition $5$ in \cite{1}, $P_{n,n_0}(\tau)\subset \mathcal {HA}_n(\tau)$. In particular each map in $P_{n,n_0}(\tau)$ has an attracting simple periodic orbit of period $n+N$.

The boxes $P_{n,n_0}(\tau)$ move holomorphically in $\tau$. Let 
$$
P_{n,n_0}\left(\mathbb D^T\right)=\cup_{\tau\in\mathbb D^T}P_{n,n_0}(\tau).
$$
We continue now the construction inductively in each $P_{n,n_0}\left(\mathbb D^T\right)$. Observe that each $P_{n,n_0}\left(\mathbb D^T\right)$ is an unfolding of the maps  in $P_{n,n_0}\left(\left(-1,1\right)^T\right)$ with a strong homoclinic tangency. Now we proceed exactly as in the proof of Theorem $B$ in \cite{1} and we create consecutive generations of nested sets  $P^k_{n^{(k)},n^{(k)}_0}\left(\mathbb D^T\right)$.
The required lamination is defined as 
 $$L_F=\bigcap_g\bigcup_{\underline n\in\mathfrak N^g}\mathcal P^g_{n^{(g)},n_0^{(g)}}\left(\mathbb D^T\right)$$ 
 and it consists of maps with infinitely many attracting periodic points of arbitrarily high periods.
 \end{proof}


\begin{thebibliography}{9}
 \bibitem {7} {Afra\u{\i}movi\v{c}, V. S. and \v{S}ilnikov, L. P.}, \emph {The singular sets of {M}orse-{S}male systems}, {Trudy Moskov. Mat. Ob\v{s}\v{c}.},
  {28},
     {1973},
     {181--214}.
	
\bibitem{1} {M. Benedicks, M. Martens, and L. Palmisano},
     \emph{Newhouse Laminations},
   {arXiv},
 {1811.00617},
    {(2018)}.
    
       \bibitem{Bi}{Biebler, S\'ebastien}, \emph{Persistent homoclinic tangencies and infinitely many sinks for residual sets of automorphisms of low degree in $\mathbb C^3$} {arXiv:1611.02011}.

  \bibitem{BrKo}  {Bronstein, I. U. and Kopanski\u\i , A. Ya.},
    \emph{Smooth invariant manifolds and normal forms},
    {World Scientific Series on Nonlinear Science. Series A:
              Monographs and Treatises},
    {7},
  {World Scientific Publishing Co., Inc., River Edge, NJ},
       {1994},
    {xii+384}.
    
    \bibitem{Bu} {Buzzard, Gregery T.},
     \emph{Infinitely many periodic attractors for holomorphic maps of
              {$2$} variables},
   {Ann. of Math. (2)},
  {145},
     {1997},
    {2},
    {389--417}.
    
     \bibitem{CLM}{De Carvalho, Andres, Lyubich, Michael and Martens, Marco}, \emph{Renormalization in the H\'enon family. I. Universality but non-rigidity}, {J. Stat. Phys.}, {121}, {2005}, {5-6}, {611--669}.
    \bibitem{6} {Forn\ae ss, John Erik and Sibony, Nessim},
      {Some open problems in higher dimensional complex analysis and
              complex dynamics},
    {Publ. Mat.},
 {45},
     {2001},
    {2},
    {529--547}.
      \bibitem {5} {Gavosto, Estela Ana}, \emph{Attracting basins in {${\bf P}^2$}},
   {J. Geom. Anal.},
 {8},
     {1998},
  {3},
     {433--440}.
   

      
  \bibitem {IlaYak} {Ilyashenko, Yu. S. and Yakovenko, S. Yu.}, \emph{Finitely smooth normal forms of local families of
              diffeomorphisms and vector fields},
   {Uspekhi Mat. Nauk}, {46},
      {1991},
    {1(277)},
     {3--39, 240}.  
     
     \bibitem{Newhouse}{Newhouse, Sheldon E.},
      \emph{Diffeomorphisms with infinitely many sinks},
   {Topology}, {13},
      {1974},
     {9--18}.
     
    
		\bibitem{PT} {Palis, Jacob and Takens, Floris}, \emph{Hyperbolicity and sensitive chaotic dynamics at homoclinic bifurcations}, {Cambridge Studies in Advanced Mathematics}, {35}, {Fractal dimensions and infinitely many attractors}, {Cambridge University Press, Cambridge},  {1993}.	
		
	\bibitem{Pa} {Palmisano, Liviana}, 
	\emph{Coexistence of non periodic attractors}, 
	{arXiv},
	{1903.01446},
	{(2019)}.
\end{thebibliography}

\end{document}